\documentclass[10 pt,reqno]{amsart}
\theoremstyle{definition}
\usepackage{indentfirst, amssymb,amsmath,amsthm}
\usepackage{csquotes}
\usepackage{hyperref}
\newtheorem*{theoA}{Theorem A}
\newtheorem*{theoB}{Theorem B}
\newtheorem*{theoC}{Theorem C}
\newtheorem*{theoD}{Theorem D}
\newtheorem*{theoE}{Theorem E}
\newtheorem*{theoF}{Theorem F}

\newtheorem{theo}{Theorem}[section]
\newtheorem{lem}{Lemma}[section]

\newtheorem{exm}{Example}[section]
\newtheorem{defi}{Definition}[section]

\newtheorem{question}{Question}[section]

\newcommand{\be}{\begin{equation}}
\newcommand{\ee}{\end{equation}}
\newcommand{\beas}{\begin{eqnarray*}}
\newcommand{\eeas}{\end{eqnarray*}}
\newcommand{\bea}{\begin{eqnarray}}
\newcommand{\eea}{\end{eqnarray}}

\numberwithin{equation}{section}
\begin{document}
\title[Periodicity of Transcendental Entire Functions Sharing Set with their Shifts]{Periodicity of Transcendental Entire Functions Sharing Set with their Shifts}
\date{}
\author[S. Roy and R. Sinha]{Soumon Roy$^{1}$ and  Ritam Sinha$^{2}$}
\date{}
\address{$^{1}$ \textbf{Soumon Roy}, Nevannlina Lab, Ramakrishna Mission Vivekananda Centenary College, Rahara,
West Bengal 700 118, India.}
\email{rsoumon@gmail.com}
\address{$^{2}$ \textbf{Ritam Sinha},  Nevannlina Lab, Ramakrishna Mission Vivekananda Centenary College, Rahara, West Bengal 700 118, India.}
\email{sinharitam23@gmail.com}
\maketitle
\let\thefootnote\relax
\footnotetext{2020 AMS Mathematics Subject Classification: 30D35}
\footnotetext{Key words and phrases: Hyper order, transcendental, entire function, Shift}
\begin{abstract}
This paper aims to study the periodicity of a transcendental entire function of hyper-order less than one. For a transcendental entire function of hyper order less than one and a non-zero complex constant $c$, $\mathfrak{f} (z) \equiv \mathfrak{f} (z + c)$ if they share a certain set with weight two.
\end{abstract}

\section{Introduction}
The main tool we used to prove our main results is value distribution  theory \cite{05}.\\
Let $E$ denote any set of positive real numbers with finite Lebesgue measure, where $E$ may vary in each occurrence. We use $M(\mathbb{C})$ to denote the field of all meromorphic functions defined in $\mathbb{C}$. Let $\mathfrak{f}\in M(\mathbb{C})$ and $ S\subset \mathbb{C} \cup \{\infty\} $ be a non-empty set with distinct elements. We set
$$ E_{\mathfrak{f}}(S) = \bigcup_{a \in S} \{z:\mathfrak{f}(z)-a=0\},$$ where a zero of $\mathfrak{f}-a$ with multiplicity $m$ counts $m$ times in $ E_{\mathfrak{f}}(S)$. Let $ \overline{E}_{\mathfrak{f}}(S)$ denote the collection of distinct elements in $ E_{\mathfrak{f}}(S)$.\\
Let $\mathfrak{g}\in M(\mathbb{C})$. We say that two function $\mathfrak{f}$ and $\mathfrak{g}$ share the set $S$ CM (res. IM) if $ E_{\mathfrak{f}}(S)=  E_{\mathfrak{g}}(S)$ (resp., $ \overline{E}_{\mathfrak{f}}(S) =  \overline{E}_{\mathfrak{g}}(S)$).\\
The order and the hyper-order of $\mathfrak{f}$, respectively denoted by $\rho(\mathfrak{f})$ and $\rho_2(\mathfrak{f})$, are defined as
$$\rho_(\mathfrak{f}):=\limsup_{r\to\infty}\frac{\log T(r,f)}{\log r},~~\text{and}~~\rho_{2}(\mathfrak{f}):=\limsup_{r\to\infty}\frac{\log \log T(r,f)}{\log r}.$$
For any non-constant meromorphic function $\mathfrak{f}$, we denote by $S(r,\mathfrak{f})$, any quantity satisfying $$ S(r,\mathfrak{f})=o(T(r,\mathfrak{f}))~~~\text{as}~~ r \rightarrow \infty, r \notin E. $$
By \emph{shift} of $\mathfrak{f}(z)$, we mean $\mathfrak{f}(z+c)$, where $c$ is a non-zero complex constant. Now, we recall some well-known definitions from the literature.
\begin{defi} (\cite{14})
Let $ a\in \mathbb{C} \cup \{\infty\}$. For a positive integer $k$, we define
\begin{enumerate}
\item[(i)] $N_{k)}(r,a; \mathfrak{f})$ as the counting function of $a$- points of $\mathfrak{f}$ with multiplicity $\leq k$,\\
\item[(ii)] $N_{(k}(r,a; \mathfrak{f})$ as the counting function of $a$- points of $\mathfrak{f}$ with multiplicity $\geq k$.
\end{enumerate}
\end{defi}
Similarly, the reduced counting function $\overline{N}_{k)}(r,a; \mathfrak{f})$ and $\overline{N}_{(k}(r,a; \mathfrak{f})$ are defined.
\begin{defi}(\cite{08})
Let $k\in \mathbb{N}$. We define $N_{k}(r,0; \mathfrak{f})$ as the counting function of zeros of $\mathfrak{f}$, where a zero of $\mathfrak{f}$ with multiplicity $q$ is counted $q$ times if $ q\leq k$, and is counted $k$ times if $q>k$.
\end{defi}
\begin{defi} (\cite{04})
A non-constant monic polynomial $P(z)$ is said to be a strong uniqueness polynomial for meromorphic function, if $P(\mathfrak{f})\equiv \alpha P(\mathfrak{g})$ implies $\mathfrak{f} \equiv \mathfrak{g}$ for any two non-constant meromorphic functions $\mathfrak{f}$ and $\mathfrak{g}$ and $\alpha$ is a non-zero constant.
\end{defi}
In 1977, Rubel and Yang \cite{17} first investigated the uniqueness of an entire function concerning its first derivatives and proved the following result:
\begin{theoA}
Let $\mathfrak{f}$ be a non-constant entire function and $a$, $b$ be two distinct finite complex values. If $ \mathfrak{f}(z)$ and $ \mathfrak{f'}(z)$ share $a$, $b$ CM, then $\mathfrak{f}(z) \equiv \mathfrak{f'}(z)$.
\end{theoA}
In 1979, Mues and Steinmetz \cite{12} improved Theorem A by replacing CM sharing by IM sharing.
\begin{theoB}
Let $\mathfrak{f}$ be a non-constant entire function and $a$, $b$ be two distinct finite complex values. If $ \mathfrak{f}(z)$ and $ \mathfrak{f'}(z)$ share $a$, $b$ IM, then $\mathfrak{f}(z) \equiv \mathfrak{f'}(z)$.
\end{theoB}
In 2011, Heittokanges et al. \cite{07} proved a similar analogue of Theorem A concerning their shifts.
\begin{theoC}
Let $\mathfrak{f}$ be a non-constant entire function of finite order, let $c$ be a nonzero finite complex value, and let $a$, $b$ be two distinct finite complex values. If $ \mathfrak{f}(z)$ and $ \mathfrak{f}(z+c)$ share $a$, $b$ CM, then $\mathfrak{f}(z) \equiv \mathfrak{f}(z+c)$.
\end{theoC}
In the same year, Qi \cite{13} further improved the Theorem C.
\begin{theoD}
Let $\mathfrak{f}$ be a non-constant entire function of finite order, let $c$ be a nonzero finite complex value, and let $a$, $b$ be two distinct finite complex values. If $ \mathfrak{f}(z)$ and $ \mathfrak{f}(z+c)$ share $a$, $b$ IM, then  $\mathfrak{f}(z) \equiv \mathfrak{f}(z+c)$.
\end{theoD}
In 2010, Zhang \cite{18, 19} considered the uniqueness between $\mathfrak{f}(z)$ and its shift $\mathfrak{f}(z+c)$ when they share two sets and obtain the following result:
\begin{theoE} \cite{18}
Let $c \in \mathbb{C}$, $S_1=\{1,\omega,\omega^2,\ldots \omega^{n-1}\}$ and $ S_2=\{\infty\}$. Suppose $\mathfrak{f}(z)$ is a non-constant meromorphic function of finite order such that $E_{\mathfrak{f}(z)}(S_j, \infty)=E_{\mathfrak{f}(z+c)}(S_j, \infty)$ for $j=1,2$ with $n\geq 4$. Then $\mathfrak{f}(z)\equiv t \mathfrak{f}(z+c)$, where $t^{n}=1$.
\end{theoE}




\begin{theoF} \cite{19}
Let $n\geq 5$ be an integer and  $a$, $b$ be two non-zero complex numbers such that the equation $\omega^n +a \omega^{n-1} +b=0$ has no multiple roots. Let $S=\{\omega~|~ \omega^n +a \omega^{n-1} +b=0\}$ and  $c$ be any non-zero complex number. If, for a  non-constant entire function $\mathfrak{f}$ of finite order, $E_{\mathfrak{f}(z)}(S) = E_{\mathfrak{f}(z+c)}(S)$, then $\mathfrak{f}(z) \equiv \mathfrak{f}(z+c).$
\end{theoF}
\begin{question}
  It is natural to ask whether the set $S$ in Theorem F can be replaced by another finite set.
\end{question}
In this note, we try to answer the above question.

\section{Periodicity of transcendental entire functions}
Let $l$ be a non-negative integer or infinity. For $a\in \mathbb{C} \cup \{\infty\}$, we denote by $E_l(a; \mathfrak{f})$, the set of all $a$- points of $\mathfrak{f}$, where an $a$- point of multiplicity $m$ is counted $m$ times if $m \leq l$ and $l+1$ times if $m> l$. \\
If for two non-constant meromorphic functions $\mathfrak{f}$ and $\mathfrak{g}$, we have $$ E_l(a; \mathfrak{f})= E_l(a; \mathfrak{g}),$$ then we say that  $\mathfrak{f}$ and $\mathfrak{g}$ share the value $a$ with weight $l$. Thus  the IM (resp. CM) sharing correspond to the weight $0$ (resp. $\infty$). The idea of weighted sharing was first introduced in \cite{09}. \\
 Let  $S \subset \mathbb{C} \cup \{\infty\}$. We define $E_{\mathfrak{f}}(S,l)$ as $$ E_{\mathfrak{f}}(S,l)= \bigcup_{a\in S} E_l(a; \mathfrak{f}),$$ where $l$ is a non-negative integer or infinity. Clearly  $E_{\mathfrak{f}}(S)=E_{\mathfrak{f}}(S,\infty)$.
\medbreak
Let
\begin{equation}\label{bc1}
 Q(z)=(z-a_1)(z-a_2)\ldots(z-a_n)
\end{equation} be a \emph{monic} and \emph{strong uniqueness polynomial} of degree $n$ in $\mathbb{C}[z]$ with $a_i\neq a_j$, $1\leq i,j\leq n$. Assume that $Q(z)-Q(0)$ have $m_1$ simple zeros and $m_2$ multiple zeros. Further suppose that $k$ be the number of distinct zeros of $Q'(z)$. Take
\begin{equation}\label{bc2}
  S=\{a_1,a_2,\ldots,a_n\}
\end{equation}
\begin{theo} \label{th1} Let $\mathfrak{f}$ be a transcendental entire function with hyper order less than one and $c$ be a nonzero complex number. If $\mathfrak{f}$ and $\mathfrak{f}(z+c)$ share the set $S$ (defined in \ref{bc2}) with weight $2$ and  $n> \max\{2k+2, m_1+m_2\}$ with $m_1+m_2>2$ and $k\geq 2$, then $\mathfrak{f}(z) \equiv \mathfrak{f}(z+c).$
\end{theo}
\begin{exm}\cite{BCM}
Consider the modified Frank-Reinders polynomial $$P(z)=\frac{(n-1)(n-2)}{2}z^n-n(n-2)z^{n-1}+\frac{n(n-1)}{2}z^{n-2}-c,$$ where $c\in \mathbb{C}\setminus\{0, 1, \frac{1}{2}\}$ and $n\geq 7$. Now,
$$P(z)-P(0)=z^{n-2}\cdot\left(\frac{(n-1)(n-2)}{2}z^2-n(n-2)z+\frac{n(n-1)}{2}\right).$$ Thus here $m_1+m_2=3$. Also, note that $$P'(z)=\frac{n(n-1)(n-2)}{2}z^{3}(z-1)^2.$$
Thus, here, $k=2$. Moreover, $P$ is strong uniqueness polynomial when $n\geq 6$ (see, Theorem 1.2, \cite{BCM}). Consider the zero set of  $P(z)$ as $S$.\\
For any  transcendental entire function $\mathfrak{f}$  with hyper order less than one, and for any a nonzero complex number $c$, if $\mathfrak{f}$ and $\mathfrak{f}(z+c)$ share the set $S$ with weight $2$, then $\mathfrak{f} \equiv \mathfrak{f}(z+c).$
\end{exm}
\medbreak
Let $l$ be a positive integer or infinity. For $a\in \mathbb{C} \cup \{\infty\}$, we denote by $E_{l)}(a; \mathfrak{f})$, the set of all $a$- points of $\mathfrak{f}$, whose multiplicities are not greater than $l$ and each such $a$ -points are counted according to multiplicity.  \par If for two non-constant meromorphic functions $\mathfrak{f}$ and $\mathfrak{g}$, we have $$ E_{l)}(a; \mathfrak{f})= E_{l)}(a; \mathfrak{g}),$$ then we say that  $\mathfrak{f}$ and $\mathfrak{g}$ share the value $a$ with ``weak weight $l$".  \par Let  $S \subset \mathbb{C} \cup \{\infty\}$. We put  $$ E_{l)}(S; \mathfrak{f})= \bigcup_{a\in S} E_{l)}(a; \mathfrak{f}),$$ where $l$ is a positive integer or infinity.
\begin{theo} \label{th2}
 Let $\mathfrak{f}$ be a transcendental entire function with hyper order less than one and $c$ be a nonzero complex number. If $\mathfrak{f}$ and $\mathfrak{f}(z+c)$ share the set $S$ (defined in \ref{bc2}) with weak weight $3$ and  $n> \max\{2k+2, m_1+m_2\}$ with $m_1+m_2>2$ and $k\geq 2$, then $\mathfrak{f}(z) \equiv \mathfrak{f}(z+c).$
\end{theo}
To establish our results, we require the following lemmas:
\begin{lem}\label{3.1} (\cite{03})
Let $\mathfrak{f}(z)$ be a non-constant transcendental meromorphic function with hyper-order less than one, and let $c$ be a non-zero finite complex number. Then $$ m\left(r,\frac{\mathfrak{f}(z+c)}{\mathfrak{f}(z)}\right)=S(r,\mathfrak{f})~~\text{and}~~ m\left(r,\frac{\mathfrak{f}(z)}{\mathfrak{f}(z+c)}\right)=S(r,\mathfrak{f}).$$
\end{lem}
\begin{lem}\label{3.2}(\cite{20}, Lemma 6)
Let $\mathfrak{f}(z)$ be a meromorphic function with hyper-order less than one, and let $c$ be a non-zero finite complex number. Then
$$T(r,\mathfrak{f}(z+c)) =T(r,\mathfrak{f})+ S(r,\mathfrak{f})~~\text{and}~~N(r,\mathfrak{f}(z+c)) =N(r,\mathfrak{f})+ S(r,\mathfrak{f}).$$
\end{lem}
\begin{lem}\label{3.3}\cite{yi}
  Let $f$ be a non-constant meromorphic function. Then $$N(r, 0; f^{(k)})\leq  k\overline{N}(r,\infty; f) + N(r, 0; f) + S(r, f).$$
\end{lem}
\begin{proof}[\textbf{Proof of the Theorem \ref{th1}}]

Let $\mathfrak{f}$ and $\mathfrak{f}(z+c)$ be two transcendental entire functions share the set $S=\{a_1,a_2,\ldots,a_n\}$ with weight $2$. Given that $Q(z)=(z-a_1)(z-a_2)\ldots(z-a_n)$, where $a_i\neq a_j$, $1\leq i,j\leq n$. Now, we define  $$ M(z):= \frac{1}{Q(\mathfrak{f}(z))}$$ and $$ N(z):= \frac{1}{Q(\mathfrak{f}(z+c))}.$$
Let $S(r)$ be any function $S: (0,\infty)\rightarrow \mathbb{R}$ satisfying $S(r)=o(T(r,\mathfrak{f})+T(r,\mathfrak{f}(z+c)))$ for $r\rightarrow \infty$ outside a set of finite Lebesgue measure. Next we consider the function $$\psi(z):= \frac{M''(z)}{M'(z)}- \frac{N''(z)}{N'(z)}.$$
First, we assume $\psi\not\equiv 0$.  Since $\psi(z)$ can be written as $$ \psi(z)=\frac{N'(z)}{M'(z)}   \left(\frac{M'(z)}{N'(z)}\right)',$$ so, all poles of $\psi$ are simple. Thus
\begin{equation} \label{4.1}
 N(r,\infty;M|=1)= N(r,\infty;N|=1)\leq N(r,0; \psi).
\end{equation}
Now by the first fundamental theorem and the lemma of logarithmic derivative, and Lemmas \ref{3.1} and \ref{3.2}, (\ref{4.1}) can be written as
\begin{equation} \label{4.2}
N(r,\infty;M|=1)= N(r,\infty;N|=1)\leq N(r,\infty; \psi) +S(r)
\end{equation}
Now, we observe that the possible poles of $\psi$ can come from\\ (i) Poles of $M$ and $N$. \\ (ii) Zeros of $M'$ and $N'$.\\
Also,
$$ M'(Z)=- \frac{\mathfrak{f}'(z)Q'(\mathfrak{f}(z))}{(Q(\mathfrak{f}(z)))^2}, N'(Z)=- \frac{\mathfrak{f}'(z+c)Q'(\mathfrak{f}(z+c))}{(Q(\mathfrak{f}(z+c)))^2}.$$
Assume that $\{b_1, b_2, \ldots, b_k\}$ be the $k$-distinct zeros of $Q'(z)$. Since $\mathfrak{f}(z)$ and $\mathfrak{f}(z+c)$ share the set $S$ with weight $2$, so, we can write
\begin{eqnarray} \label{4.3}
\nonumber  N(r,\infty; \psi) &\leq& \sum_{j=1}^{k}\left( \overline{N}(r,b_j;\mathfrak{f}(z))+ \overline{N}(r,b_j;\mathfrak{f}(z+c))\right)+ \overline{N}_{0}(r,0;\mathfrak{f'}(z))+ \overline{N}_{0}(r,0;\mathfrak{f'}(z+c))\\
&& +  \overline{N}(r,\infty;\mathfrak{f}(z))+ \overline{N}(r,\infty;\mathfrak{f}(z+c))+ \overline{N}_{*}(r,\infty;M, N),
\end{eqnarray}
where $\overline{N}_{0}(r,0;\mathfrak{f'}(z))$ denotes the reduced counting function of zeros of $\mathfrak{f'}$, which are not zeros of $ \Pi_{i=1}^{n} (\mathfrak{f}-a_i)\Pi_{j=1}^{k} (\mathfrak{f}-b_j)$. Similarly, $\overline{N}_{0}(r,0;\mathfrak{f'}(z+c))$ is defined. Since $\mathfrak{f}(z)$ and $\mathfrak{f}(z+c)$ share $S$ with weight $2$, so,
\begin{eqnarray} \label{4.4}
&& \nonumber \overline{N}(r,\infty;M|\geq2)+ \overline{N}_{0}(r,0;\mathfrak{f'}(z+c))+\overline{N}_{*}(r,\infty; M, N)\\
 \nonumber & \leq & \overline{N}_{0}(r,0;Q(\mathfrak{f}(z+c))|\geq2)+\overline{N}_{0}(r,0;\mathfrak{f'}(z+c))+ \overline{N}_{0}(r,0;Q(\mathfrak{f}(z+c))|\geq3)\\
\nonumber & \leq & N(r,0;\mathfrak{f'}(z+c))\\
& \leq & N(r,0;\mathfrak{f}(z+c))+S(r,\mathfrak{f}(z+c)).
\end{eqnarray}
Now, by the 2nd fundamental theorem, Lemma \ref{3.2} and from (\ref{4.2}), (\ref{4.3}) and (\ref{4.4}), we have
\begin{eqnarray} \label{4.5}
\nonumber && (n+k-1)T(r,\mathfrak{f})\\
\nonumber & \leq & \overline{N}(r,\infty;\mathfrak{f}(z)) + \overline{N}(r,0;Q(\mathfrak{f}))+\sum_{j=1}^{k}\left( \overline{N}(r,b_j;\mathfrak{f}(z)\right) -\overline{N}_{0}(r,0;\mathfrak{f'}(z))+S(r,\mathfrak{f})\\
\nonumber & \leq & \overline{N}(r,\infty;M\geq 2)+N(r,\infty;\psi)+\sum_{j=1}^{k}\left( \overline{N}(r,b_j;\mathfrak{f}(z)\right) -\overline{N}_{0}(r,0;\mathfrak{f'}(z))+S(r,\mathfrak{f})\\
\nonumber & \leq &  2\sum_{j=1}^{k}\left( \overline{N}(r,b_j;\mathfrak{f}(z)\right)+\sum_{j=1}^{k}\left( \overline{N}(r,b_j;\mathfrak{f}(z+c)\right)+\overline{N}_{0}(r,0;\mathfrak{f'}(z+c))\\
\nonumber&&+ \overline{N}_{*}(r,\infty; M, N)+ \overline{N}(r,\infty;M\geq 2)+S(r,\mathfrak{f})\\
\nonumber & \leq & 2k T(r,\mathfrak{f}(z)) + k T(r,\mathfrak{f+c})+ N(r,0; \mathfrak{f}(z+c))+ S(r,\mathfrak{f})\\
\nonumber& \leq & (3k+1)T(r,\mathfrak{f}) +S(r,\mathfrak{f}),
\end{eqnarray}
which is impossible as $ n>2k+2$. Therefore, $\psi\equiv 0$. Then by integration, we have $$\frac{1}{Q(\mathfrak{f}(z))}\equiv \frac{d_0}{Q(\mathfrak{f}(z+c))}+d_1.$$
Here, we consider two cases.\\
\textbf{Case 1} Assume that $d_1\neq 0$. Then above expression can be written as
$$Q(\mathfrak{f}(z))\equiv \frac{Q(\mathfrak{f}(z+c))}{d_0+ d_1Q(\mathfrak{f}(z+c))}.$$
Thus $$\overline{N}\left(r,-\frac{d_0}{d_1};Q(\mathfrak{f}(z+c))\right) \leq \overline{N}(r, \infty ;Q(\mathfrak{f}(z)))= \overline{N}(r, \infty ;\mathfrak{f}(z)).$$
Since $Q(z)-Q(0)$ has $m_1$ simple zeros and $m_2$ multiple zeros, so we can assume that $$Q(z)-Q(0) = (z-b_1)(z-b_2)\ldots (z-b_{m_1})(z-c_1)^{l_1}(z-c_2)^{l_2}\ldots (z-c_{m_2})^{l_{m_2}},$$ where $l_i \geq 2$ for $1\leq i\leq m_2$. Moreover $l_i< n$ as $Q'(z)$ has at least two zeros.\par
If $Q(0)\neq -\frac{d_0}{d_1}$, then by the first and second fundamental theorems, we have
\begin{eqnarray*}
 nT(r,\mathfrak{f}(z+c)) +O(1) & = & T(r,Q(\mathfrak{f}(z+c))) \\
 & \leq & \overline{N}\left(r, \infty ;Q(\mathfrak{f}(z+c))\right)+ \overline{N}\left(r, Q(0);Q(\mathfrak{f}(z+c))\right)\\
 & + &\overline{N}\left(r,-\frac{d_0}{d_1};Q(\mathfrak{f}(z+c))\right)+S(r,\mathfrak{f}(z+c))\\
 & \leq & (m_1+m_2)T(r,\mathfrak{f}(z+c))+S(r,\mathfrak{f}(z+c)),
\end{eqnarray*}
which is impossible as $n>m_1+m_2$.\\ If $Q(0)= -\frac{d_0}{d_1}$, then
$$Q(\mathfrak{f}(z))\equiv \frac{Q(\mathfrak{f}(z+c))}{ d_1\left(Q(\mathfrak{f}(z+c))-Q(0)\right)}.$$
Thus every zero of $\mathfrak{f}(z+c)-b_j$ $(1\leq j\leq m_1)$ has multiplicity at least $n$, and every zero of $\mathfrak{f}(z+c)-c_i$ $(1\leq i\leq m_2)$ has multiplicity at least $2$.\\
Thus by the second fundamental theorem, we have
\begin{eqnarray*}
&& (m_1+m_2-1)T(r,\mathfrak{f}(z+c))\\
  &\leq & \overline{N}\left(r, \infty ;\mathfrak{f}(z+c)\right)+ \sum_{j=1}^{m_1}\overline{N}\left(r, b_j ;\mathfrak{f}(z+c)\right)+ \sum_{i=1}^{m_2}\overline{N}\left(r, c_i ;\mathfrak{f}(z+c)\right)+ S(r,\mathfrak{f}(z+c))\\
& \leq & \frac{1}{n} \sum_{j=1}^{m_1}N\left(r, b_j ;\mathfrak{f}(z+c)\right)+ \frac{1}{2}\sum_{i=1}^{m_2}N\left(r, c_i ;\mathfrak{f}(z+c)\right)+ S(r,\mathfrak{f}(z+c))\\
&\leq & \left(\frac{m_1}{2}+\frac{m_2}{2}\right)T(r, \mathfrak{f}(z+c))+S(r, \mathfrak{f}(z+c)),
\end{eqnarray*}
which is not possible.\medbreak
\textbf{Case 2} Thus we assume that $d_1 = 0$. Then $ Q(\mathfrak{f}(z)) \equiv d_0 Q(\mathfrak{f}(z+c))$. Since $Q$ is strong uniqueness polynomial, thus
 $$\mathfrak{f}(z) \equiv \mathfrak{f}(z+c).$$
Hence the theorem is proved.
\end{proof}
\begin{proof}[\textbf{Proof of the Theorem \ref{th2}}]
This case is similar of Theorem \ref{th1}. Thus we omit the details.
\section{Compliance with Ethical Standards}

\begin{enumerate}
\item[•] Conflicts of Interest: The authors declare that they have no conflict of interest.
\item[•] Data Availability Statement: Data Sharing is not applicable to this article.
\item[•] Authors received no specific grant/funding for the research, authorship or publication
of this article.
\item[•] This article does not contain any studies with human participants or animals
performed by any of the authors.
\end{enumerate}

\end{proof}

\end{document}